\documentclass[12pt]{article}
\usepackage{amsmath}
\usepackage{amsfonts}
\usepackage{setspace}
\usepackage{xcolor}
\usepackage{amsthm}
\usepackage{authblk}

\newtheorem{theorem}{Theorem}[section]
\newtheorem{lemma}[theorem]{Lemma}
\newtheorem{cor}[theorem]{Corollary}

\newtheorem{prop}[theorem]{Proposition}

\newtheorem{remark}[theorem]{Remark}
\newtheorem{definition}[theorem]{Definition}

\newtheorem{question}[theorem]{Question}
\newtheorem{example}[theorem]{Example}
\newtheorem{conj}[theorem]{Conjecture}

\newcommand{\D}{\mathbb{D}}

\newcommand{\T}{\mathbb{T}}
\newcommand{\Pol}{\mathcal{P}}
\newcommand{\Hol}{\operatorname{Hol}}

\begin{document}

\title{Some Remarks on Shanks-type Conjectures}

\author{Christopher Felder\footnote{\textit{Department of Mathematics, Indiana University, Bloomington, Indiana 47405.}\\ Email: \textit{cfelder@iu.edu}}}
\date{ }

\maketitle

\begin{abstract}
We discuss the zero sets of two-variable polynomials as they relate to an approximation problem in the Hardy space on the bidisk. 
\end{abstract}


\section{Introduction}
Let $\D := \left\{ z \in \mathbb{C} : |z|<1\right\}$ denote the open unit disk and, for $d$ a positive integer,  let  $\D^d = \times_{j = 1}^d\D $ denote the polydisk of dimension $d$. The present work is concerned with the Hardy spaces on these domains, which are defined as 
\[
H^2(\D^d) := \left\{ f \in \Hol(\D^d) : \sup_{0 \le r < 1} \int_{\T^d} \left| f(rz) \right|^2 \, dm(z) < \infty \right\}, 
\]
where $\Hol(\D^d)$ is the set of holomorphic functions on the polydisk and $dm$ is normalized Lebesgue measure on the $d$-dimensional torus $\T^d$. It is well known that functions in $H^2(\D^d)$ have radial limits almost everywhere on $\T^d$ and that $H^2(\D^d)$ is a closed subspace of $L^2(\T^d, dm)$. 

Given $f \in H^2(\D^d)$, we are interested in polynomials $p$ which solve the minimization problem
\[
\inf_{p \in \Pol_n}\|pf - 1\|, 
\]
where $\Pol_n$ is the collection of analytic polynomials of a given (multi)degree $n$. For fixed $n$ and $f$, the polynomial minimizing the above quantity is known as the $n$-th \textit{optimal polynomial approximant} (OPA) to $1/f$. In particular, we are interested in the zero sets of these polynomials. We refer to results which describe the \textit{location} of the zeros of OPAs as \textit{Shanks-type results}. Primarily, we will be concerned with Shanks-type results which describe whether or not the zeros of an OPA are inside the (open or closed) polydisk.

For example, it is well known that when $d = 1$, OPAs cannot vanish in the closed unit disk. A proof of this can be found in \cite[Theorem~4.2]{MR3614926} but was known to engineers working in digital filter design in the 1970s (see \cite{MR4244844, MR4410997} and the references therein). 

From an engineering perspective, there has been a desire to extend this result to the case $d>1$. This prompted Shanks, Treitel, and Justice  \cite{1162358} to conjecture that if $f \in H^2(\D^2)$ is any polynomial, then the OPAs to $1/f$ cannot vanish inside the closed bidisk.
Within a few years, this conjecture was proved to be false by Genin and Kamp \cite{Genin1975}.
After disproving the conjecture of Shanks et al., the authors in \cite{Genin1975} went on to construct a method for producing polynomials with OPAs having zeros in the bidisk \cite{1454852}. However, their counterexamples to Shanks' conjecture were for the optimal approximants of functions with zeros in the bidisk. In turn, they conjectured the following, which, after almost fifty years, is still unresolved:
\begin{conj}[Weak Shanks Conjecture, \cite{Genin1975}]
If $f \in H^2(\D^2)$ is a polynomial with no zeros in the \textbf{closed} bidisk, then the OPAs to $1/f$ cannot vanish inside the open bidisk.
\end{conj}
Apart from their examples, it is unclear what additional intuition guided the authors in \cite{Genin1975} to make this conjecture.

The purpose of this note is to communicate some remarks on Shanks-type conjectures in general. Unfortunately, we do not provide a resolution to the Weak Shanks Conjecture, but do provide some contributions and observations in the direction of resolving the conjecture. In short, the difficulty of this conjecture is in determining a precise relationship between the zero set of a function and the zero sets of its OPAs.

The outline of the paper is as follows: 
\begin{itemize}
\item We will provide notational conventions and more thoroughly discuss background in Section \ref{background}.

\item In Section \ref{factorizations}, we show that that if an OPA has a one-variable factor, then that factor must vanish outside of the closed disk (Theorem \ref{p-n-factor}). We also show that if $f \in H^2(\D^2)$ is a polynomial of a certain form, then the OPAs to $1/f$ are, up to a multiplicative constant, the same as the OPAs to the reciprocal of the \textit{reflection} of $f$ (Theorem \ref{hound}). 

\item In Section \ref{shanks-theorems}, we prove that if $f \in H^2(\D^2)$ is a function of one variable or a one-variable function `in disguise,' then the OPAs to $1/f$ cannot vanish in the closed bidisk (Theorem \ref{disguise}). We also provide a result which characterizes precisely when a function with no zeros in the bidisk has OPAs with no zeros in the bidisk (Theorem \ref{det-form}). 

\item Section \ref{w-inner} concerns weakly-inner functions.
\end{itemize}


\section{Notation and Background}\label{background}
For relevant background on Hilbert spaces on polydisks we suggest \cite{Rudin} as a general reference, but will list a few elementary facts here: 

\begin{itemize}
\item The space $H^2(\D^d)$ is a Hilbert space; for 
\[
f(z_1, \ldots, z_d) = \sum_{j_1, \ldots, j_d\ge 0} a_{j_1,\ldots j_d}z_1^{j_1} \dots z_d^{j_d}
\]
and 
\[
g(z_1, \ldots, z_d) = \sum_{j_1, \ldots, j_d\ge 0} b_{j_1,\ldots, j_d}z_1^{j_1} \ldots z_d^{j_d},
\]
both elements of $H^2(\D^d)$,
we have 
\[
\langle f, g \rangle = \sum_{j_1, \ldots, j_d\ge 0} a_{j_1,\ldots j_d} \overline{b_{j_1,\ldots, j_d}}.
\]
We mention that functions in $H^2(\D^d)$ have radial limits almost everywhere on  $\T^d$ and this inner product may also be expressed as the inner product on $L^2(\T^d, dm)$:
\[
\langle f, g \rangle =  \int_{\T^d} f \overline{g} \ dm,
\]
where $dm$ is normalized Lebesgue measure. 
\item Further, $H^2(\D^d)$ is a \textit{reproducing kernel Hilbert space} on $\D^d$; this means, for every $\alpha = (\alpha_1, \ldots, \alpha_d) \in \D^d$, the linear functional of point evaluation 
\[
f \mapsto f(\alpha)
\]
is bounded. By the Riesz Representation Theorem, for each $\alpha \in \D^d$ there is a unique element $k_{\alpha} \in H^2(\D^d)$ so that for any $f \in H^2(\D^d)$, we have
\[
f(\alpha) = \langle f, k_{\alpha} \rangle. 
\]
The element $k_{\alpha}$ is called the \textit{reproducing kernel} at $\alpha$ and can be verified to be the function
\[
k_{\alpha}(z_1, \ldots, z_d) = \prod_{j=1}^d\frac{1}{1 - \overline{\alpha_j}z_j}.
\]
\item The set of polynomials in $d$ variables is dense in $H^2(\D^d)$. 
\end{itemize} 

We will be concerned with the cases $d=1, 2$. In this setting, we will use the indeterminates $z$ and $w$ instead of $z_1$ and $z_2$, respectively. 
We will also abuse notation and use $zf$ or $wf$ to mean the action of the multiplication operator on the function $f\in H^2(\D^2)$ induced by multiplication of these variables.

\subsection{Optimal Polynomial Approximants}
Let us begin by recalling the degree lexicographic order of $d$-variable  polynomials, where monomials are ordered by increasing total degree and ties are broken lexicographically. This is the usual order when $d=1$. For two variables, we use the notation
\[
\chi_0 = 1, \chi_1 = z, \chi_2 = w, \chi_3 = z^2, \chi_4 = zw,  \chi_5 = w^2, \ldots 
\]
and so on. 
For $n$ a positive integer, if we define $\Pol_n := \operatorname{span}\{\chi_0, \ldots, \chi_n \}$, then we have a chain of subspaces
\[
\Pol_0 \subset \Pol_1 \subset \dots \subset \Pol_n \subset \Pol_{n+1} \subset \cdots . 
\]
A standard functional analysis argument then shows, given $f \in H^2(\D^d)$, the sequence of orthogonal projections
\[
\Pi_n \colon H^2(\D^d) \to  \Pol_n f : = \{pf : p\in \Pol_n\}
\]
converges, in the strong operator topology, to the orthogonal projection from $H^2(\D^d)$ onto the closure of all polynomial multiples of $f$ in $H^2(\D^d)$ (see, e.g., \cite[Proposition~2.4]{Felder}). We will not use this property of the ordering, but mention that it is useful in certain other contexts, including limits of optimal polynomial approximants, which we formally discuss now. 

\begin{definition}[OPA]
Let $f \in H^2(\D^d)$. Given a non-negative integer $n$, the \textit{$n$-th optimal polynomial approximant (OPA)} to $1/f$ in $H^2(\D^d)$ is the polynomial solving the minimization problem
\[
\min_{q \in \Pol_n}\|qf - 1\|_p.
\]
This polynomial is given by the orthogonal projection of $1$ onto the subspace $\Pol_n f$ (therefore uniquely exists), and will be denoted by
\[
p_n^*[f].
\]
\end{definition}
An elementary linear algebra exercise shows that the coefficients of an OPA, say,  $p_n^*[f] = \sum_{j=0}^n a_j \chi_k$, can be recovered from the linear system
\[
\left(\langle \chi_j f, \chi_k f \right)_{0 \le, j,k \le n} \left(a_0, \ldots, a_n\right)^T 
= \left(\overline{f(0)}, 0, \ldots, 0\right)^T.
\]
We refer to this linear system as the \textit{optimal system} (with respect to $f$, of degree $n$). Note that the Gram matrix above is of full rank when $f$ is not identically zero. Consequently, $p_n^*[f] \equiv 0$ if and only if $f$ vanishes at the origin; later, we make the assumption $f(0) \neq 0$ in order to avoid this case. 
See \cite{MR4410997, MR4327435} for more background on OPAs in several variables. 

Let us provide an example to illustrate the methods at play. 

\begin{example}
Let $f(z, w) = 1 - z - w$. For $n = 0$ and $p_0^*[f] = a_0$, we have
\[
\|f\|^2 a_0 = \langle 1, f \rangle, 
\]
which gives
\[
p_0^*[f](z,w) = \frac{\langle 1, f \rangle}{\|f\|^2} = \frac{1}{3}.
\]
For $n=1$ and $p_1^*[f](z,w) = a_0 + a_1z$, we can recover $a_0, a_1$ as the solution to the system
\[
\begin{bmatrix}
\|f\|^2 & \langle f, zf \rangle\\
\langle zf, f \rangle & \|zf\|^2
\end{bmatrix}
\begin{bmatrix}
a_0 \\ 
a_1
\end{bmatrix}
= \begin{bmatrix}
\langle 1, f \rangle \\ 
0
\end{bmatrix}.
\]
Making the appropriate computations, this yields
\[
\begin{bmatrix}
3 & -1\\
-1 & 3
\end{bmatrix}
\begin{bmatrix}
a_0 \\ 
a_1
\end{bmatrix}
= \begin{bmatrix}
1 \\ 
0
\end{bmatrix}.
\]
In turn, we have 
\[
p_1^*[f](z,w) = \frac38 + \frac18 z.
\]
Carrying on in a similar manner, the coefficients of $p_2^*[f](z,w) = a_0 + a_1z + a_2 w$ can be recovered via the system
\[
\begin{bmatrix}
\|f\|^2 & \langle f, zf \rangle & \langle f, wf \rangle\\
\langle zf, f \rangle & \|f\|^2 & \langle zf, wf \rangle\\
\langle wf, f \rangle & \langle wf, zf \rangle & \|f\|^2
\end{bmatrix}
\begin{bmatrix}
a_0 \\ 
a_1\\
a_2
\end{bmatrix}
= \begin{bmatrix}
\langle 1, f \rangle \\ 
0\\
0
\end{bmatrix},
\]
which is
\[
\begin{bmatrix}
3 & -1 & -1\\
-1 & 3 & 1\\
-1 & 1 & 3
\end{bmatrix}
\begin{bmatrix}
a_0 \\ 
a_1\\
a_2
\end{bmatrix}
= \begin{bmatrix}
1 \\ 
0 \\
0
\end{bmatrix},
\]
and gives
\[
p_2^*[f](z,w) = \frac25 + \frac{1}{10} z + \frac{1}{10} w.
\]       
\end{example}

\subsection{The Shanks Conjecture}
Let us state a theorem in the one-variable Hardy space.
As previously mentioned, there are several proofs of this theorem, one of which can be found in \cite[Theorem~4.2]{MR3614926}, which has a surprising connection with the zeros of orthogonal polynomials on the unit circle. We will provide a different proof here that uses only the Fundamental Theorem of Algebra and the Cauchy-Schwarz Inequality.

\begin{theorem}\label{OPA-bnd}
Let $f \in H^2(\D)$ with $f(0) \neq 0$. For each $n\ge 0$, the zeros of $p_n^*[f]$ lie outside the closed unit disk. 
\end{theorem}

\begin{proof}
The case $n = 0$ is trivial, so let us proceed assuming $n>0$.

Let $w \in \mathbb{C}$, suppose $p_n^*[f](w) = 0$, and put $p_n^*[f](z) = (z-w)\tilde{p}(z)$. 
As $p_n^*[f] f$ is the orthogonal projection of 1 onto $\bigvee\{f, zf, \ldots, z^nf\}$, we have
\[
1 - p_n^*[f]f \perp z\tilde{p}f.
\]
In turn, this yields
\begin{align*}
&0 = \langle 1 - p_n^*[f],  z\tilde{p}f \rangle \\
\implies & 0 = \langle 1 - (z- w)\tilde{p} f,  z\tilde{p}f \rangle \\
\implies & w\langle \tilde{p} f,  z\tilde{p}f \rangle = \langle z\tilde{p}f - 1,  z\tilde{p}f \rangle \\
\implies & w = \frac{\langle z\tilde{p}f - 1,  z\tilde{p}f \rangle}{\langle \tilde{p} f,  z\tilde{p}f \rangle}\\
\implies & w = \frac{\|z\tilde{p}f\|^2 - \langle 1, z\tilde{p}f\rangle}{\langle \tilde{p}f, z\tilde{p}f\rangle}\\
\implies & w = \frac{\|\tilde{p}f\|^2}{\langle \tilde{p}f, z\tilde{p}f\rangle}\ ,
\end{align*}
where the last implication above follows from the fact that multiplication by $z$ is an isometry on $H^2(\D)$ and $1 = k_0$ is the reproducing kernel at the origin. 
Cauchy-Schwarz then yields
\[
|w| > \frac{\|\tilde{p}f\|^2}{\|\tilde{p}f\| \|z\tilde{p}f\|} = 1, 
\]
where we note the inequality is strict as $\tilde{p}f$ and $z\tilde{p}f$ are not collinear. 
\end{proof}

It is natural to ask if a similar result holds in $H^2(\D^2)$. As previously mentioned,  Shanks, Treitel, and Justice \cite{1162358}, in studying digital filter design, conjectured the following:
\begin{conj}[Strong Shanks Conjecture]
\textit{If $f \in H^2(\D^2)$ is any polynomial, then the OPAs to $1/f$ cannot vanish inside the closed bidisk.}
\end{conj}

Again, this conjecture was quickly proved to be false \cite{Genin1975}. A simplified counterexample can be found in \cite[Example~20]{MR4410997}. Although this particular conjecture is false, it is natural to wonder if a similar result may be true for some other subclass of Hardy space functions. In particular, one may ask:

\begin{question}\label{non-zero}
For which functions $f \in H^2(\D^2)$ is it true that the OPAs to $1/f$ don't vanish in the (closed) bidisk? Is there a function-theoretic characterization of such functions?
\end{question}

If $f\in H^2(\D^2)$ has OPAs which do not vanish in the closed bidisk, we will say that $f$ satisfies a \textit{Shanks-type result}.


\section{Factorizations}\label{factorizations}
Let us begin by exploring what happens under the assumption that an OPA has a one-variable factor. 

\begin{prop}\label{p-n-factor}
Let $f\in H^2(\D^2)$ with $f(0) \neq 0$. If 
\[
p_n^*[f](z,w) = (z - \alpha)\tilde{p}(z,w),
\]
then
\[
|\alpha|>1.
\]
Similarly, if $p_n^*[f](z,w) = (w - \beta)\tilde{q}(z,w)$, then $|\beta| > 1$.
\end{prop}

\begin{proof}
We mimic the argument in the proof of Theorem \ref{OPA-bnd}. Recall that $\langle p_n^*[f], qf \rangle = 0$ for all polynomials $q \in \Pol_n$ with $q(0) = 0$ (this can be seen by noting that $p_n^*[f]f$ is the orthogonal projection of $1$ onto $\Pol_n f$ and $1$ is the reproducing kernel at the origin in $H^2(\D^2)$).
In particular, taking 
\[
q(z, w) = z\tilde{p}(z,w), 
\]
which is an element of $\Pol_n$, we have
\begin{align*}
0 & = \langle p_n^*[f]f, z\tilde{p}f \rangle \\
&= \langle (z - \alpha)\tilde{p}f, z\tilde{p}f \rangle \\
&= \langle z \tilde{p}f, z\tilde{p}f \rangle - \alpha \langle \tilde{p}f, z\tilde{p}f \rangle.
\end{align*}
In turn, we have 
\[
\alpha = \frac{ \langle z \tilde{p}f, z\tilde{p}f \rangle}{ \langle \tilde{p}f, z\tilde{p}f \rangle}.
\]
Noting that multiplication by $z$ (or $w$) is an isometry on $H^2(\D^2)$, we have 
\begin{align*}
|\alpha| &= \frac{ \| z\tilde{p}f \|^2}{ \left|\langle \tilde{p}f, z\tilde{p}f \rangle\right|} \\
&> \frac{ \| z\tilde{p}f \|^2}{ \|\tilde{p}f\| \| z\tilde{p}f\|} \\
& = 1,
\end{align*}
where we have used the Cauchy-Schwarz inequality with the fact that $\tilde{p}f$ and $z\tilde{p}f$ are not collinear. The same argument holds when $z$ and $w$ are switched. 
\end{proof}

\begin{remark}
This result may be regarded as potential evidence for certain Shanks-type results. 
We also note that this result has an immediate corollary; any function $f \in H^2(\D^2)$ satisfies a Shanks-type result for the first-degree OPA. We will discuss this again following Theorem \ref{det-form}.
\end{remark}

We will now make some observations when $f$ has factorizations of a certain form. First, we need a lemma. 

\begin{lemma}\label{opt-equiv}
Let $f, g \in H^2(\D^2)$ with $f(0) \neq 0$ and $g(0) \neq 0$. If, for $0\le j,k \le n$, we have
\[
\langle \chi_j f, \chi_k f \rangle = \langle \chi_j g, \chi_k g \rangle,
\]
then there exists a constant $c$ so that
\[
p_n^*[f] = c \cdot p_n^*[g].
\]
\end{lemma}

\begin{proof}
Consider the optimal systems
\[
\left(\langle \chi_j f, \chi_k f \right)_{0 \le, j,k \le n} \left(a_0, \ldots, a_n\right)^T 
= \left(\overline{f(0)}, 0, \ldots, 0\right)^T
\]
and 
\[
\left(\langle \chi_j g, \chi_k g \right)_{0 \le, j,k \le n} \left(b_0, \ldots, b_n\right)^T 
= \left(\overline{g(0)}, 0, \ldots, 0\right)^T.
\]
Notice that, for $c = \overline{f(0/g(0)}$, we have 
\begin{align*}
\left(\langle \chi_j f, \chi_k f \right)_{0 \le, j,k \le n} \left(a_0, \ldots, a_n\right)^T 
&= c \cdot \left(\overline{g(0)}, 0, \ldots, 0\right)^T \\
& = c \left(\langle \chi_j g, \chi_k g \right)_{0 \le, j,k \le n} \left(b_0, \ldots, b_n\right)^T \\
& = c \left(\langle \chi_j f, \chi_k f \right)_{0 \le, j,k \le n} \left(b_0, \ldots, b_n\right)^T.
\end{align*}
As each Gram matrix above is invertible, we see 
\[
\left(a_0, \ldots, a_n\right) = c \left(b_0, \ldots, b_n\right).
\]
The result then follows.
\end{proof}

We will employ this lemma but first need to recall a definition. A function $f\in H^2(\D^2)$ is said to be \textit{inner} in $H^2(\D^2)$ if 
\[
|f| = 1  \ \ \text{a.e. on $\T^2$}.
\]
A more thorough discussion of inner functions will occur later but we provide an initial observation here first.

\begin{prop}
Let $f \in H^2(\D^2)$ with $f(0) \neq 0$. For any factorization of the form
\[
f(z,w) = \theta(z, w) \tilde{f}(z,w),
\]
with $\theta$ an inner function, and for every $n\ge0$, there exists a constant $c$ so that
\[
p_n^*[f] = c \cdot p_n^*[\tilde{f}].
\]
\end{prop}

\begin{proof}
Using the fact that $\theta$ induces an isometric multiplication operator on $H^2(\D^2)$, it is elementary to check, for all $j,k \ge0$, that
\[
\langle \chi_j f, \chi_k f \rangle = \langle \chi_j \tilde{f}, \chi_k \tilde{f} \rangle. 
\]
Invoking Lemma \ref{opt-equiv} then gives the result.
\end{proof}

We can say something similar for other types of factorizations. Let us first establish some notation. 

\begin{definition}[Reflection]
A two-variable polynomial has bi-degree $(n, m)$ if it is degree $n$ in $z$ and degree $m$ in $w$. For such a polynomial, we define the \textit{reflection} of $p$ as 
\[
\overleftarrow{p}(z, w) := z^n w^m \overline{p\left(\frac{1}{\overline{z}}, \frac{1}{\overline{w}}\right)}.
\]
\end{definition}
Notice that, as an operator, this reflection is a densely defined isometry on $H^2(\D^2)$; it is elementary to check that $\|\overleftarrow{p} \| = \| p \|$ for any polynomial $p$.

\begin{theorem}\label{hound}
Suppose $f \in H^2(\D^2)$ with $f(0) \neq 0$. For any factorization of the form
\[
f(z,w) = q(z,w) \tilde{f}(z,w),
\] 
with $q$ a polynomial such that $\overleftarrow{q}(0)\neq 0$, and any $n\ge0$, there exists a constant $c$ so that 
\[
p_n^*[f] = c \cdot p_n^*[ \overleftarrow{q} \tilde{f}].
\]
\end{theorem}

\begin{proof}
It suffices to notice that for any polynomial $q$, we have $|q|^2 = |\overleftarrow{q}|^2$ on $\T^2$. Thus, for any $0\le j,k \le n$, we have
\begin{align*}
\langle \chi_j f, \chi_k f \rangle &= \langle \chi_j q \tilde{f}, \chi_k q \tilde{f} \rangle\\
&= \langle \chi_j \overleftarrow{q} \tilde{f}, \chi_k \overleftarrow{q} \tilde{f} \rangle.
\end{align*}
So long as $\overleftarrow{q}(0) \neq 0$, we may invoke Lemma \ref{opt-equiv} to conclude the result. 
\end{proof}

\begin{remark}
Heuristically, this theorem says that the OPAs to $1/f$ do not care about the location of certain polynomial factors of $f$. This could be, perhaps, strong motivation to search for a counterexample to the Weak Shanks Conjecture.
In particular,  Theorem \ref{hound} implies that if $f$ is a polynomial so that $\overleftarrow{f}(0) \neq 0$ (i.e., if $f$ is of degree $(n, m)$, then the coefficient of $z^nw^m$ in $f$ is non-zero) , then 
\[
p_n^*[f] = p_n^*[\overleftarrow{f}].
\]
The function $f$ could have no zeros in the closed bidisk, while $\overleftarrow{f}$ has a plethora of zeros inside $\mathbb{D}^2$(!). 

This observation could also be of use when considering known counterexamples to the Strong Shanks Conjecture. For example, what can be said of the counterexamples of Genin and Kamp \cite{1454852} with respect to factorizations? Can their algorithm be used to construct a function with OPAs vanishing in $\D^2$, but for which the function can be factored as the product of two polynomials, one of which vanishes only in the bidisk and the other which vanishes only outside the closed bidisk? If so, the above observation could be used to produce a counterexample to the Weak Shanks Conjecture. 
\end{remark}

We will end this section with an observation about constant OPAs. 

\begin{prop}
Let $f, g, h \in H^2(\D^2)$, each non-zero at the origin, and with $g(z, w) = g(z)$ and $h(z,w) = h(w)$. If $f = gh$, then $p_0^*[f] = p_0^*[g] \cdot p_0^*[h]$. 
\end{prop}

\begin{proof}
Recall that
\[
p_0^*[f] = \frac{\overline{f(0)}}{\|f\|^2}.
\]
Since $\overline{f(0)} = \overline{g(0)h(0)}$, it suffices to show that $\|f\|^2 = \|g\|^2\|h\|^2$. Put $g(z) = \sum_{j \ge 0}a_j z^j$ and $h(w) = \sum_{k \ge 0}b_k w^k$. Then
\begin{align*}
\|f\|^2 &= \left\| \left(\sum_{j \ge 0}a_j z^j \right) \left(\sum_{k \ge 0}b_k w^k \right)   \right\|^2\\
&= \left\|  \sum_{n \ge 0} \sum_{\ell\ge 0}a_\ell z^\ell b_{n-\ell} w^{n-\ell}  \right\|^2 \\
&= \sum_{n=0}^\infty \sum_{\ell=0}^n \left| a_\ell b_{n-\ell} \right|^2\\
&= \left(\sum_{j \ge 0} |a_j|^2\right) \left(\sum_{k \ge 0} |b_k|^2\right)\\
&= \|g\|^2 \|h\|^2.
\end{align*}
\end{proof}

Notice that the third equality in the above proof holds only because $g$ and $h$ are functions of \textit{different} variables. 
This begs a natural question.
\begin{question}\label{factor}
Are there any circumstances where the OPAs to $1/f = 1/gh$ factor as the OPAs to $1/g$ and $1/h$?
\end{question}

If this question has an affirmative answer, this may be a tool for producing a counterexample to the Weak Shanks Conjecture, especially if the factorization of $f$ might hold with two one-variable factors. In this case, some traction might be gained, since the varieties at hand would be much easier to understand.


\section{Shanks-Type Results}\label{shanks-theorems}

We begin by providing a class of functions for which a Shanks-type theorem holds (c.f. Question \ref{non-zero}), but first need a lemma.

\begin{lemma}\label{perp-factor}
Let $f \in H^2(\D^2)$ with $f(0)\neq 0$. If $p_n^*[f]$ can be expressed as
\[
p_n^*[f](z,w) = q(z,w) + r(z,w), \  \ \ q,r \in \mathcal{P}_n
\]
with $r(0) = 0$ and $qf \perp rf$, then $r \equiv 0$ and $p_n^*[f](z,w) = q(z,w)$. 
\end{lemma}

\begin{proof}
Observe:
\begin{align*}
    \|p_n^*[f]f - 1\|^2 &= \|(q+r)f - 1\|^2\\
    &= \|(q+r)f\|^2 - 2\operatorname{Re}\{ \langle (q+r)f, 1\rangle \} + 1\\
    &= \|qf\|^2 + 2\operatorname{Re}\{ \underbrace{\langle qf, rf\rangle}_{=0} \} + \|rf\|^2 - 2\operatorname{Re}\{ \langle qf, 1\rangle + \underbrace{\langle rf, 1\rangle}_{=0} \} + 1\\
    &= \|qf-1\|^2 + \|rf\|^2 \\
    &\ge \|qf-1\|^2.
\end{align*}
By definition, this tells us that $p_n^*[f] = q$. 
\end{proof}

\begin{theorem}\label{disguise}
Let $h \in H^2(\D)$. If $f(z,w)= h(\chi_k)$ for some $k\ge0$, then, for each $n\ge 0$, $p_n^*[f]$ is zero-free in the closed bidisk.  
\end{theorem}
Note that $f(z,w)= h(\chi_k)$ is an element of $H^2(\D^2)$; the coefficients are square-summable.
\begin{proof}
Begin by collecting all terms of $p_n^*[f]$ which are a power of $\chi_k$, including the constant term; call the sum of these terms $q(\chi_k)$. Now, put
\[
p_n^*[f](z,w)  = q(\chi_k) + r(z,w),
\]
and note that $r(0) = 0$ by construction.
Notice that $\langle qf, rf \rangle = 0$, so by Lemma \ref{perp-factor}, we can conclude that $p_n^*[f](z,w) = q(\chi_k)$. Now, we can factor $p_n^*[f]$ as 
\[
p_n^*[f] = c \prod_{j=1}^M \left(\chi_k - \alpha_j\right),
\]
for some constant $c$ and $M$ a positive integer. 
Putting $\chi_k = z^\ell w^m$ and arguing in the same manner as in the proof of Theorem \ref{OPA-bnd}, we see that $p_n^*[f]$ vanishes only when $|z^\ell w^m| > 1$. However, if $(z, w) \in \overline{\D}^2$, then $|z^\ell w^m| \le 1$. From this, we conclude that $p_n^*[f]$ is zero-free in the closed bidisk.
\end{proof}

\begin{remark}
This theorem tells us that there are plenty of functions for which a Shanks-type theorem holds. On the other hand, it also tells us that if searching for a counterexample to the Weak Shanks Conjecture, one must work in ``true'' polynomials of two variables; it will not suffice to consider a polynomial of one variable, or a one-variable polynomial ``disguised" in two variables.
\end{remark}

We now give an equivalent formulation for a function with no zeros in the bidisk to have OPAs with no zeros in the bidisk. 

\begin{theorem}\label{det-form}
Let $f \in H^2(\D^2)$ be zero-free in the bidisk. Then the OPA $p_n^*[f]$ has no zeros in the bidisk if and only if, for every $\alpha \in \D^2$, the matrix
\[
\begin{bmatrix}
\langle k_\alpha, f \rangle & \langle f, \chi_1 f\rangle & \dots & \langle f, \chi_n f\rangle\\
\langle k_\alpha, \chi_1 f \rangle & \langle \chi_1 f, \chi_1 f\rangle & \dots & \langle \chi_1 f, \chi_n f\rangle\\
\vdots & \vdots & \ddots & \vdots \\
\langle k_\alpha, \chi_n f \rangle & \langle \chi_n f, \chi_1 f\rangle & \dots & \langle \chi_n f, \chi_n f\rangle
\end{bmatrix}
\]
is invertible.
\end{theorem}

\begin{proof}
For $\alpha \in \D^2$, let $q_\alpha f$ be the orthogonal projection of the kernel $k_\alpha$ onto $\Pol_n f$. Notice that 
\[
p_n^*[f](\alpha)f(\alpha) = \langle p_n^*[f]f, q_\alpha f \rangle = \overline{q_\alpha(0)f(0)}.
\]
In turn, $p_n^*[f]$ vanishes at $\alpha$ if and only if $q_\alpha$ vanishes at the origin. This occurs if and only if the constant term in $q_\alpha$ is zero. Putting $q_\alpha = \sum_{j=0}^n b_j\chi_j$, routine linear algebra (Cramer's rule) gives
\[
b_0 = \frac{\det G_0}{\det G},
\]
where $G$ is the Gram matrix  $(\langle \chi_j f, \chi_k f \rangle)_{0\le j,k \le n}$ and $G_0$ is formed by replacing the first column of $G$ by 
\[
[\langle k_\alpha, f \rangle, \langle k_\alpha, \chi_1 f \rangle, \ldots, \langle k_\alpha, \chi_n f \rangle]^T. 
\]
Thus, we have $b_0$ is zero if and only if 
\[
\det G_0 = \left|
\begin{matrix}
\langle k_\alpha, f \rangle & \langle f, \chi_1 f\rangle & \dots & \langle f, \chi_n f\rangle\\
\langle k_\alpha, \chi_1 f \rangle & \langle \chi_1 f, \chi_1 f\rangle & \dots & \langle \chi_1 f, \chi_n f\rangle\\
\vdots & \vdots & \ddots & \vdots \\
\langle k_\alpha, \chi_n f \rangle & \langle \chi_n f, \chi_1 f\rangle & \dots & \langle \chi_n f, \chi_n f\rangle
\end{matrix}\right|
=0.
\]
The result then follows. 
\end{proof}

\begin{remark}
We know, from Proposition \ref{p-n-factor}, that if we are to search for a counterexample to the Weak Shanks Conjecture, we must consider the case when the OPA has degree $n\ge 2$. In the case $n=2$, we have considered various implications of Theorem \ref{det-form}. However, even in this simplest case, the analysis is non-trivial and it is unclear how to relate the determinantal expression to the zero set of the function $f$. This difficulty demonstrates the crux of the Weak Shanks Conjecture.
Nonetheless, Theorem \ref{det-form} could potentially be a powerful tool for more effectively exploring the conjecture.
\end{remark}


\section{Weakly Inner Functions}\label{w-inner}
A function $f \in H^2(\D^2)$ is said to be \textit{inner} in $H^2(\D^2)$ if 
\[
|f|= 1 \ \ \text{a.e. on $\T^2$},
\]
and \textit{weakly-inner} if 
\[
\langle f, \chi_k f \rangle = 0 \ \ \text{for all $k\ge1$}.
\]
Notice that every weakly inner function is inner. 

Theorem \ref{det-form} gives an immediate corollary for certain weakly inner functions (however, we will shortly see that this result is trivialized by a result of Sargent and Sola \cite{MR4410997}). 
\begin{cor}\label{sing-int}
Let $f \in H^2(\D^2)$ be zero-free in the bidisk. If $f$ is weakly-inner, then the OPA $p_n^*[f]$ has no zeros in the open bidisk. \end{cor}

\begin{proof}
If $f$ is weakly inner, then for any $\alpha \in \D^2$, we have 
\begin{align*}
& \det 
\begin{bmatrix}
\langle k_\alpha, f \rangle & \langle f, \chi_1 f\rangle & \dots & \langle f, \chi_n f\rangle\\
\langle k_\alpha, \chi_1 f \rangle & \langle \chi_1 f, \chi_1 f\rangle & \dots & \langle \chi_1 f, \chi_n f\rangle\\
\vdots & \vdots & \ddots & \vdots \\
\langle k_\alpha, \chi_n f \rangle & \langle \chi_n f, \chi_1 f\rangle & \dots & \langle \chi_n f, \chi_n f\rangle
\end{bmatrix}\\
= &
\det
\begin{bmatrix}
\overline{f(\alpha)} & 0 & \dots &0\\
\langle k_\alpha, \chi_1 f \rangle & \langle \chi_1 f, \chi_1 f\rangle & \dots & \langle \chi_1 f, \chi_n f\rangle\\
\vdots & \vdots & \ddots & \vdots \\ 
\langle k_\alpha, \chi_n f \rangle & \langle \chi_n f, \chi_1 f\rangle & \dots & \langle \chi_n f, \chi_n f\rangle
\end{bmatrix}\\[.5cm]
&= \overline{f(\alpha)}\det(\langle \chi_j f, \chi_k f \rangle)_{1\le j,k \le n}). 
\end{align*}
As $\overline{f(\alpha)}$ is non-zero (by hypothesis) and the Gram matrix $(\langle \chi_j f, \chi_k f \rangle)_{1\le j,k \le n})$ is of full rank, the above determinant must be non-zero. The result then follows from Theorem \ref{sing-int}.
\end{proof}

Although this result may seem interesting, as it could be interpreted as Shanks-type theorem for ``singular" inner functions on the bididsk, it is really just a statement about constants: 

\begin{prop}{\cite[Proposition~7]{MR4410997}}
If $f \in H^2(\D^2)$ is weakly inner then 
\[
p_n^*[f] = p_0^*[f] \ \ \text{for all $n \ge 0$}.
\]
That is, weakly inner functions have OPAs which are all constant. 
\end{prop}

We note that this result can be strengthened to an \textit{if and only} statement; it follows from the same approach in \cite[Lemma~3.7]{Felder}. We include Corollary \ref{sing-int} in order to stress that one must be careful about hypothesis when looking for Shanks-type results.

There is a related result which deals with functions which satisfy $\langle f, \chi_k f \rangle = 0$ for $k$ in some subset of natural numbers (such functions will not always be constants). Before stating this result, we need a definition.

\begin{definition}
For each $k, n \ge 0$, let $\pounds(k,n)$ be the index $j$ so that $(\chi_k)^n = \chi_j$. 
\end{definition}
For example, we have $\pounds(0, n) = 0$ for all $n\ge 0$ (since $\chi_0^n = 1^n = \chi_0$), $\pounds(1, 2) = 3$ (since $\chi_1^2 = z^2 = \chi_3$), and $\pounds(2, 2) = 5$ (since $\chi_2^2 = w^2 = \chi_5$). 

Using this notation, we show that if $f \in H^2(\D^2)$ is a one-variable function or one-variable function in disguise, then $p_n^*[f]$ will plateau for certain consecutive values of $n$.

\begin{prop}
Let $h \in H^2(\D^2)$ with $h(0) \neq 0$.
If $f(z, w) = h(\chi_k)$ for some $k \ge 0$, then, for each $j \ge 0$ and each $N = \pounds(k, j), \ldots, \pounds(k, j+1) - 1$, we have 
\[
p_{\pounds(k, j)}^*[f] = p_{N}^*[f].
\]
\end{prop}

\begin{proof}
It suffices to notice that if $f(z, w) = \sum_{\ell \ge 0}a_\ell\chi_k^\ell$, then $\chi_n f$ and  $\chi_m f$ are orthogonal for any $n$ and $m$ for which both $n \neq \pounds(k,j)$ for some $j \ge 0$ and $m \neq \pounds(k, j')$ for some $j'\ge 0$. 
Considering the optimal system for $f$ of degree $\pounds(k, j+1) - 1$ then immediately gives the result. 
\end{proof}

\begin{example}
For $f(z,w) = 1 - zw$, we have 
\begin{align*}
& p_0^*[f](z,w) = p_1^*[f](z,w) = p_2^*[f](z,w)  = \frac12, \\
& p_3^*[f](z,w) = p_4^*[f](z,w) = \dots = p_{11}^*[f](z,w) =  \frac23 + \frac13 zw, \\
& p_{12}^*[f](z,w) = p_{13}^*[f](z,w) = \dots = p_{21}^*[f](z,w) =  \frac34 + \frac12 zw + \frac14(zw)^2. \\
\end{align*}
\end{example}

Other interesting questions can be asked about (weakly) inner functions. For example, inner functions are bounded-- is the same true for \textit{weakly} inner functions? The answer to this turns out to be \textit{no}. 

\begin{prop}\label{unbdd}
There are weakly-inner functions which are not bounded on the bidisk.
\end{prop}

\begin{proof}
Let $\mathcal{M} \subset H^2(\D^2)$ be any shift invariant subspace and let $\varphi: = \Pi_{\mathcal{M}}1$ be the orthogonal projection of $1$ onto $\mathcal{M}$. Since $\mathcal{M}$ is shift-invariant, we have $\chi_k \varphi \in \mathcal{M}$ for any $k\ge0$. In turn, for any $k \ge 1$, we have,
\[ 
\langle \varphi, \chi_k \varphi \rangle =\langle \Pi_{\mathcal{M}}1, \chi_k \varphi \rangle = \langle 1, \chi_k \varphi \rangle = 0.
\]
Hence, $\varphi$ is weakly inner. However, there are shift invariant subspaces of $H^2(\D^2)$ which do not possess any bounded functions (see \cite{Rudin}).
Therefore, if $\mathcal{M}$ is one of these subspaces, then $\varphi$ will be unbounded. 
\end{proof}

Up to this point, we have left existing literature, and results therein, largely untapped-- especially the literature originating from outside of the immediate OPA community. 
We conclude, in an appendix, with a few remarks about other possible tools to address the Weak Shanks Conjecture.

\section*{Appendix: Determinantal Representations}

A two-variable polynomial is called \textit{stable} if it has no zeros in the \textit{open} bidisk and \textit{strongly stable} if it has no zeros in the \textit{closed} bidisk. 
Using this nomenclature, one can state the Weak Shanks Conjecture as:
\[
\text{Any strongly stable polynomial has OPAs which are stable.}
\]
Stable polynomials have certain determinantal representations (see, e.g. \cite{MR3441374}).
The following theorem will be taken as a black box. 

\begin{theorem}
A polynomial $p(z,w)$ is stable if and only if it can be expressed as 
\[
p(z, w) = \alpha \det(I - CD),
\]
where $\alpha$ is a constant, $C$ is an $n\times n$ matrix which is contractive, and $D$ is a diagonal matrix of the form $\operatorname{diag}(z, \ldots, z, w, \ldots, w)$. 

Further, $p$ is strongly stable if and only if the above expression holds with $C$ being a strict contraction.
\end{theorem}

\begin{example}
The polynomial $p(z,w) = 2 - z - w$ is stable but not strongly stable ($p$ is non-zero on the open bidisk but $p(1,1) = 0$), and has the determinantal represntation
\[
p(z,w) = 2 \det \left(
\begin{pmatrix}
1 & 0 \\
0 & 1
\end{pmatrix}
- 
\begin{pmatrix}
1/2 & 1/2 \\
1/2 & 1/2
\end{pmatrix}
\begin{pmatrix}
z & 0 \\
0 & w
\end{pmatrix} 
\right).
\]
Note that 
\[
\left\|
\begin{pmatrix}
1/2 & 1/2 \\
1/2 & 1/2
\end{pmatrix}
\right\|
=1.
\]
Compare this to the polynomial $q(z,w) = 4 - z - w$, which is strongly stable and has the determinantal representation
\[
q(z,w) = 4 \det \left(
\begin{pmatrix}
1 & 0 \\
0 & 1
\end{pmatrix}
- 
\begin{pmatrix}
1/4 & 1/4 \\
1/4 & 1/4
\end{pmatrix}
\begin{pmatrix}
z & 0 \\
0 & w
\end{pmatrix} 
\right),
\]
with
\[
\left\|
\begin{pmatrix}
1/4 & 1/4 \\
1/4 & 1/4
\end{pmatrix}
\right\|
=\frac{1}{\sqrt{2}} < 1.
\]

\end{example}

If $f \in H^2(\mathbb{D}^2)$ is strictly stable and has the determinantal representation 
\[
f = \det(I - AB),
\]
with appropriate $m \times m$ matrices, then if the Weak Shanks Conjecture is true, $p_n^*[f]$ must have determinantal representation
\[
p_n^*[f] = \det(I - CD),
\]
where $C$ is a $k\times k$ matrix which is a contraction, and $D$ is a diagonal matrix in $z$ and $w$. This says that the polynomial $p_n^*[f]f$ must have representation
\[
p_n^*[f]f  = \det(I - EF),
\]
where $E$ is a contraction (not necessarily strict, and not necessarily $m\times m$ or $k\times k$(!)) and $F$ is a diagonal matrix in $z$ and $w$. If the size of the matrices in the representations for $f$ and $p_n^*[f]$ are equal, then we have 
\begin{align*}
 \det(I - AB)\det(I - CD)
& = \det(I - AB - CD + ABCD)\\
& = \det(I - EF).
\end{align*}

Thus, one approach to understanding the Weak Shanks Conjecture is through understanding if such a determinantal factorization can always exist. Note that this works only for $f$ and $p_n^*[f]$ having the same representation size. However, we can augment any representation to arrive at a similar statement. In particular, if $p_n^*[f]$ has a $k \times k$ representation with $k < m$ (here $m$ is the representation size of $f$), i. e.
\[
p_n^*[f] = \det(I_k - GH),
\]
with $G$ a $k \times k$ contraction and $H$ a $k\times k$ diagonal in $z$ and $w$, then we can augment $I_k - GH$ as 
\[
\begin{pmatrix}
I_k - GH & 0 \\
0 & I_{m-k}
\end{pmatrix}
\]
to get an $m\times m$ matrix with the same determinant as $I_k - GH$. The upshot here is that the dimensions are now correct to make sense of the representation 
\[
p_n^*[f]f = \det\begin{pmatrix}
I_k - GH & 0 \\
0 & I_{m-k}
\end{pmatrix}
\det(I_m - AB).
\]
If $p_n^*[f]$ is such that its representation is of size $k > m$, then we can augment the representation of $f$ in a similar way. The matrix multiplication here turns out to be messy, so the approach may not be very friendly. However, it is nonetheless valid.

We end by pointing out that there is a natural matrix associated with an OPA. If $G = (\langle \chi_j f, \chi_k f)_{0 \le j,k \le n}$, then 
\[
p_n^*[f] = \left(\chi_0, \ldots, \chi_n\right)\ G^{-1}\left(\overline{f(0)}, 0 , \ldots, 0\right)^T. 
\]
Can we witness this as a determinant? Is it useful to note that since $G$ is a Gram matrix, that $G^{-1}$ must also be a Gram matrix? Could a factorization of $G^{-1}$ be useful?

\section*{Acknowledgements}
Many thanks to Alan Sola for discussion and in helping deduce Proposition \ref{unbdd} by pointing out the result in Rudin's book; this was done during a visit  in the Summer of 2023, for which the Department of Mathematics at Stockholm University provided partial support. Further thanks go to Greg Knese for helpful discussion.

\bibliography{SHANKS_PUB_v2}

\begin{thebibliography}{10}

\bibitem{MR4244844}
C.~B\'{e}n\'{e}teau and R.~Centner.
\newblock A survey of optimal polynomial approximants, applications to digital
  filter design, and related open problems.
\newblock {\em Complex Anal. Synerg.}, 7(2):Paper No. 16, 12, 2021.

\bibitem{MR3614926}
C.~B\'{e}n\'{e}teau, D.~Khavinson, C.~Liaw, D.~Seco, and A.~A. Sola.
\newblock Orthogonal polynomials, reproducing kernels, and zeros of optimal
  approximants.
\newblock {\em J. Lond. Math. Soc. (2)}, 94(3):726--746, 2016.

\bibitem{Felder}
C.~Felder.
\newblock General optimal polynomial approximants, stabilization, and
  projections of unity.
\newblock {\em Anal. Theory Appl.}, 39(4):309--329, 2023.

\bibitem{Genin1975}
Y.~Genin and Y.~Kamp.
\newblock Counterexample in the least-squares inverse stabilisation of 2d
  recursive filters.
\newblock {\em Electronics Letters}, 11(15):330 -- 331, 1975.

\bibitem{1454852}
Y.~Genin and Y.~Kamp.
\newblock Two-dimensional stability and orthogonal polynomials on the
  hypercircle.
\newblock {\em Proceedings of the IEEE}, 65(6):873--881, 1977.

\bibitem{MR3441374}
A.~Grinshpan, D.~S. Kaliuzhnyi-Verbovetskyi, V.~Vinnikov, and H.~J. Woerdeman.
\newblock Stable and real-zero polynomials in two variables.
\newblock {\em Multidimens. Syst. Signal Process.}, 27(1):1--26, 2016.

\bibitem{Rudin}
W.~Rudin.
\newblock {\em Function theory in polydiscs}.
\newblock W. A. Benjamin, Inc., New York-Amsterdam, 1969.

\bibitem{MR4327435}
M.~Sargent and A.~A. Sola.
\newblock Optimal approximants and orthogonal polynomials in several variables
  {II}: {F}amilies of polynomials in the unit ball.
\newblock {\em Proc. Amer. Math. Soc.}, 149(12):5321--5330, 2021.

\bibitem{MR4410997}
M.~Sargent and A.~A. Sola.
\newblock Optimal approximants and orthogonal polynomials in several variables.
\newblock {\em Canad. J. Math.}, 74(2):428--456, 2022.

\bibitem{1162358}
J.~Shanks, S.~Treitel, and J.~Justice.
\newblock Stability and synthesis of two-dimensional recursive filters.
\newblock {\em IEEE Transactions on Audio and Electroacoustics},
  20(2):115--128, 1972.

\end{thebibliography}
\bibliographystyle{abbrv}

\end{document}